\documentclass[11pt,leqno]{amsart}
\usepackage[utf8]{inputenc}
\usepackage{hyperref}
\usepackage{tikz}
\usepackage{amsmath,amssymb,amsthm}

\newtheorem{theorem}{Theorem}[section]
\newtheorem{lemma}[theorem]{Lemma}

\newtheorem{proposition}[theorem]{Proposition}
\newtheorem{corollary}[theorem]{Corollary}
\theoremstyle{definition}

\newtheorem{question}[theorem]{Question}
\theoremstyle{remark}
\newtheorem{remark}[theorem]{Remark}

\numberwithin{equation}{section}

\def\fnote#1{\footnote}

\def\real{{\mathbb R}}

\def\F{{\mathcal F}}
\def\R{{\mathbb R}}

\def\om{{\omega}}
\def\la{{\langle}}
\def\ra{{\rangle}}
\def\ep{{\varepsilon}}

\def\dw{{d_{\omega}}}

\def\ignora#1{}
\def\n3#1{\left\vert  \! \left\vert \! \left\vert \, #1 \, \right\vert \!
  \right\vert \! \right\vert }

\newcommand{\iten}{\ensuremath{\widehat{\otimes}_\varepsilon}}
\newcommand{\pten}{\ensuremath{\widehat{\otimes}_\pi}}

\title{On the structure of spaces of vector-valued Lipschitz functions }

\author{Luis Garc\'ia-Lirola}\thanks{First author was partially supported by the grants MINECO/FEDER MTM2014-57838-C2-1-P and Fundaci\'on S\'eneca CARM 19368/PI/14 and by the Region de Franche-Comté.}
\address{Universidad de Murcia, Facultad de Matem\'aticas. Departamento de Matem\'aticas,
30100-Murcia, Spain}\email{luiscarlos.garcia@um.es}

\author{Colin Petitjean}\thanks{Second author was partially supported by the grants ECOS-Sud program - Action number C14E06.}
\address{Universit\'e Franche-Comt\'e\\
Laboratoire de Math\'ematiques UMR 6623\\
16 route de Gray\\
25030 Besan\c con Cedex\\
France}\email{colin.petitjean@univ-fcomte.fr}

\author{Abraham Rueda Zoca}\thanks{The third author was partially supported by Junta de Andaluc\'ia Grants FQM-0199.}
\address{Universidad de Granada, Facultad de Ciencias.
Departamento de An\'{a}lisis Matem\'{a}tico, 18071-Granada, Spain} \email{ arz0001@correo.ugr.es}\urladdr{\url{https://arzenglish.wordpress.com}}



\date{}

\begin{document}

\begin{abstract}
We analyse the strong connections between spaces of vector-valued Lipschitz functions and spaces of linear continuous operators. We apply these links to study duality, Schur properties and norm attainment in the former class of spaces as well as in their canonical preduals.
\end{abstract}

\maketitle

\section{Introduction}
\bigskip

The problem of whether Lipschitz-free Banach spaces ($\mathcal F(M)$), which are canonical preduals of spaces of Lipschitz functions ($Lip(M)$), are themselves dual ones has been studied for a long time (see e.g. \cite{dal1,dal2,kalton2004,weaver}). For instance, given a compact (respectively proper) metric space $M$, it is known that $\mathcal F(M)$ is the dual of $lip(M)$ (respectively $S(M)$) under some additional assumptions on $M$ (see formal definitions below). The vector-valued version of previous spaces $\mathcal F(M,X)$ has been recently introduced in \cite{blr} as a predual of the space of vector-valued Lipschitz functions $Lip(M,X^*)$ in the spirit of the scalar version. Moreover, in \cite[Proposition 1.1]{blr}, it is proved that $\mathcal F(M,X)$ is linearly isometrically isomorphic to $\mathcal F(M)\pten X$. So, basic tensor product theory yields a canonical predual of vector-valued Lipschitz-free Banach spaces, namely the injective tensor product of the predual of each factor, whenever they exist and verify some natural conditions. On the other hand, natural vector-valued extensions of the preduals of $\mathcal F(M)$, namely $lip(M,X)$ and $S(M,X)$, have been recently considered in \cite{gr,jsv} and identified with a suitable subspace of compact operators from $X^*$ to $Lip(M)$. Consequently, in order to generalise the preduality results to the vector-valued setting, it is a natural question whether the equality $S(M,X)=S(M)\iten X$ holds. We will show that it is the case under suitable assumptions on $M$ and $X$.

The identification of vector-valued Lipschitz-free Banach spaces with a projective tensor product not only motivates the problem of analysing duality but also other properties. Some of them have been even analysed in the scalar version such as approximation properties or Schur property. As approximation properties are preserved by injective as well as by projective tensor products \cite{h}, it is straightforward that such properties on $\mathcal F(M,X)$ actually rely on the scalar case $\mathcal F(M)$. Nevertheless, it is an open problem if the Schur property is preserved by projective tensor product \cite[Remark 6]{gg}, so it is natural to wonder which conditions on $M$ and $X$ guarantee that $\mathcal F(M,X)$ enjoys the Schur property, a problem which has been previously considered in the scalar framework \cite{hlp,kalton2004}.

Finally, it is a natural question how different notions of norm-attainment in $Lip(M,X)$ are related. On the one hand, in this space there is a clear notion of norm-attainment for a Lipschitz function. On the other hand, the equality $Lip(M,X)=L(\mathcal F(M),X)$ yields the classical notion of norm-attainment considered in spaces of linear and continuous operators. So we can wonder when the previous concepts agree and, in such a case, analyse when the class of Lipschitz functions which attain its norm is dense in $Lip(M,X)$,
a problem motivated by the celebrated Bishop-Phelps theorem and recently studied in \cite{godsurvey} and \cite{kadets}.

The paper is organised as follows. In Section \ref{duality} we will explore the duality problem and get the two main results of the section. On the one hand, in Theorem \ref{predulibrevectorproper} we get that $S(M,X)^*=\mathcal F(M,X^*)$ whenever $M$ is proper satisfying $S(M)^*=\mathcal F(M)$ and either $X^*$ or $\mathcal F(M)$ has the approximation property. On the other hand, we prove in Theorem \ref{tauduality} that $lip_\tau(M,X)^*=\mathcal F(M,X^*)$ whenever either $\mathcal F(M)$ or $X^*$ has the approximation property and $(M,\tau)$ satisfies the assumptions under which $\mathcal F(M)=lip_\tau(M)^*$ in \cite[Theorem 6.2]{kalton2004}. So, previous results are vector-valued extensions of the preduality results in the real case given in \cite{dal2} and \cite{kalton2004}. In Section \ref{schur} we take advantage of the theory of tensor products to study the (hereditary) Dunford-Pettis property on $S(M,X)$ and the (strong) Schur property on $\mathcal F(M,X)$. More precisely, we prove that $S(M,X)$ has the hereditary Dunford-Pettis property and does not contain any isomorphic copy of $\ell_1$ whenever $X$ satisfies those two conditions and $M$ is a proper metric space such that $S(M)^*=\mathcal F(M)$ (Theorem 3.2). As a direct corollary,  we obtain under the same assumptions that $\mathcal F(M,X^*)$ has the strong Schur property. We end this section by extending a result of Kalton to the vector-valued setting by proving that if $M$ is uniformly discrete and $X$ has the Schur property, then $\mathcal F(M,X)$ has the Schur property (Proposition 3.4). This last result provides examples of Banach spaces with the Schur property such that their projective tensor product also enjoys this property. Furthermore, in Section \ref{na}, we will deal with the problem of norm-attainment, proving in Proposition \ref{vectorna} denseness of $NA(\mathcal F(M),X^{**})$ in $L(\mathcal F(M),X^{**})$ whenever $M$ is a proper metric space satisfying $S(M)^*=\mathcal F(M)$ and either $\mathcal F(M)$ or $X^*$ has the approximation property. Finally, in Section \ref{rema} we will pose some open problems and make some related comments.

\textbf{Notation}. Given a metric space $M$, $B(x,r)$ (respectively $\overline{B}(x,r)$) denotes the open (respectively closed) ball in $M$ centered at $x\in M$ with radius $r$. According to \cite{kalton2004}, by a \textit{gauge} we will mean a continuous, subadditive and increasing function $\omega:\mathbb R_0^+\longrightarrow \mathbb R_0^+$ verifying that $\omega(0)=0$ and that $\omega(t)\geq t$ for every $t\in [0,1]$. We will say that a gauge $\omega$ is \textit{non-trivial} whenever $\lim\limits_{t\rightarrow 0}\frac{\omega(t)}{t}=\infty$. Throughout the paper we will only consider real Banach spaces. Given a Banach space $X$, we will denote by $B_X$ (respectively $S_X$) the closed unit ball (respectively the unit sphere) of $X$. We will also denote by $X^*$ the topological dual of $X$. We will denote by $X\pten Y$ (respectively $X\iten Y)$ the projective (respectively injective) tensor product of Banach spaces. For a detailed treatment and applications of tensor products, we refer the reader to \cite{rya}. In addition, $L(X,Y)$ (respectively $K(X,Y)$) will denote the space of continuous (respectively compact) operators from $X$ to $Y$. Moreover, given topologies $\tau_1$ on $X$ and $\tau_2$ on $Y$, we will denote by $L_{\tau_1,\tau_2}(X,Y)$ and $K_{\tau_1,\tau_2}(X, Y)$ the respective subspaces of $\tau_1$-$\tau_2$ continuous operators. According to \cite[Proposition 4.1]{rya}, a Banach space $X$ is said to have the \textit{approximation property (AP)} whenever given a compact subset $K$ of $X$ and a positive $\varepsilon$ there exists a finite rank operator $S\in L(X,X)$ such that $\Vert x-S(x)\Vert<\varepsilon$ for all $x\in K$. Note that this property obviously passes through to preduals.

Given a metric space $M$ with a designated origin $0$ and a Banach space $X$, we will denote by $Lip(M,X)$ the Banach space of all $X$-valued Lipschitz functions on $M$ which vanish at $0$ under the standard Lipschitz norm 
$$\Vert f\Vert:=\sup\left\{ \frac{\Vert f(x)-f(y)\Vert}{d(x,y)}\ :\ x,y\in M, x\neq y \right\} .$$
First of all, note that we can consider any point of $M$ as an origin with no loss of generality, because the resulting Banach spaces turn out to be isometrically isomorphic. Moreover, it is known that $Lip(M,X^*)$ is a dual Banach space, with a canonical predual given by
$$\mathcal F(M,X):=\overline{span}\{\delta_{m,x}: m\in M, x\in X\}\subseteq Lip(M,X^*)^*,$$
where $\delta_{m,x}(f):=f(m)(x)$ for every $m\in M, x\in X$ and $f\in Lip(M,X^*)$ (see \cite{blr}). Furthermore, it is known that, for every metric space $M$ and every Banach space $X$, $Lip(M,X)=L(\mathcal F(M),X)$ (e.g.~\cite{jsv}) and that $\mathcal F(M,X)=\mathcal F(M)\pten X$ (see \cite[Proposition 1.1]{blr}).

We will also consider the following spaces of vector-valued Lipschitz functions.  
\begin{align*}
lip(M,X)&:=\left\{f\in Lip(M,X): \exists\ \lim\limits_{\varepsilon\rightarrow 0} \sup\limits_{0<d(x,y)<\varepsilon} \frac{\Vert f(x)-f(y)\Vert}{d(x,y)}=0\right\},\\
S(M,X)&:=\left\{ f\in lip(M,X) : \exists\ \lim\limits_{r\rightarrow \infty} \sup_{\stackrel{x \text{ or } y\notin B(0,r)}{x\neq y}} \frac{\Vert f(x)-f(y)\Vert}{d(x,y)}=0\right\}.
\end{align*}
We will avoid the reference to the Banach space when it is $\mathbb R$ in the above definitions. Finally, we will say that a subspace $S\subset Lip(M)$ \emph{separates points uniformly} if there exists a constant $c\geq 1$ such that for every $x, y\in M$ there is $f\in S$ satisfying $||f||\leq c$ and $f(x)-f(y)=d(x,y)$. Recall that if $M$ is a proper metric space then $S(M)$ separates points uniformly if, and only if, it is a predual of $\mathcal F(M)$ \cite{dal2}.

\section{Duality results on vector-valued Lipschitz-free Banach spaces} \label{duality}
\bigskip

Let $(M,d)$ be a metric space, $X$ be a Banach space and assume that there exists a subspace $S$ of $Lip_0(M)$ such that $S^* = \F(M)$. Note that basic tensor theory yields the following identification:
$$ \mathcal F(M,X^*) = \mathcal F(M) \widehat\otimes_\pi X^* = (S \iten X)^{*}$$
whenever either $\mathcal F(M)$ or $X^*$ has (AP) and either $\mathcal F(M)$ or $X^*$ has the Radon-Nikod\'ym property (RNP)(see \cite[Theorem 5.33]{rya})
However, the natural question here is when we can give a representation of a predual of $\mathcal F(M,X)$ as a subspace of $Lip(M,X^{*})$.

It has been recently proved in \cite[Theorem 5.2]{gr} that $S(M,X)$ is isometrically isomorphic to $K_{w^*,w}(X^*,S(M))$ whenever $M$ is proper. Consequently, in order to prove that $S(M,X)^*=\mathcal F(M,X^*)=\mathcal F(M)\pten X^*$ under natural assumptions on $M$, we shall begin by analysing when the equality $K_{w^*,w}(X^*,Y)=X\iten Y$ holds. In order to do that, we shall need to introduce two results.

\begin{lemma} 
Let $X, Y$ be Banach spaces. Then $T\mapsto T^*$ defines an isometry from $K_{w^*,w}(X^*, Y)$ onto $K_{w^*,w}(Y^*, X)$.
\end{lemma}

\begin{proof} Let $T\in K(X^*, Y)\cap L_{w^*,w}(X^*,Y)$. Then $T^*\in K(Y^*, X^{**})$. Moreover, given $y^* \in Y^*$, we have that $T^*(y^*)=y^*\circ T: X^*\to\real$ is weak-star continuous and thus $T^*(y^*)\in X$. Therefore $T^*\in K(Y^*, X)$. Since $T^*$ is $\sigma(Y^*, Y)-\sigma(X^{**}, X^*)$-continuous, we get $T^*\in L_{w^*,w}(Y^*,X)$. 
Conversely, if $R\in K(Y^*, X)\cap L_{w^*,w}(Y^*,X)$ then $R^*\in K(X^*, Y)\cap L_{w^*,w}(X^*,Y)$ and $R^{**} = R$.    
\end{proof}

Next proposition is well-known (see Remark 1.2 in \cite{ruessstegal}), although we have not found any proof in the literature. We include it here for the sake of completeness. 

\begin{proposition}\label{tensor}
Let $X$ and $Y$ be Banach spaces and assume that either $X$ or $Y$ has (AP). Then $K_{w^*,w}(X^*, Y)= X\iten Y$.
\end{proposition}
\begin{proof} 
By the above lemma we may assume that $Y$ has the (AP).  Clearly the inclusion $\supseteq$ holds, so let us prove the reverse one. To this aim pick $T:X^*\longrightarrow Y$ a compact operator which is $w^*-w$ continuous. We will approximate $T$ in norm by a finite-rank operator following word by word the proof of \cite[Proposition 4.12]{rya}. As $Y$ has the (AP) we can find $R:Y\longrightarrow Y$ a finite-rank operator such that $\Vert x-R(x)\Vert<\varepsilon$ for every $x\in T(B_{X^*})$, and define $S:=R\circ T$. $S$ is clearly a finite-rank operator such that $\Vert S-T\Vert<\varepsilon$. As $S$ is a finite rank operator, then $S=\sum_{i=1}^n x_i^{**}\otimes y_i$ for suitable $n\in \mathbb N, x_i^{**}\in X^{**}$ and $y_i\in Y$. Moreover $S$ is $w^*-w$ continuous. Indeed, the fact that $S$ is $w^*-w$ continuous means that, for every $y^*\in Y^*$, one has
$$y^*\circ S=\sum_{i=1}^n y^*(y_i)x_i^{**}:X^*\longrightarrow \mathbb R$$
is a weak-star continuous functional, so $\sum_{i=1}^n y^*(y_i) x_i^{**}\in X$ for each $y^*\in Y^*$. Note that an easy argument of bilinearity allows us to assume that $\{y_1,\ldots, y_n\}$ are linearly independent. Now, a straightforward application of Hahn-Banach theorem yields that, for every $i\in\{1,\ldots, n\}$, there exists $y_i^*\in Y^*$ such that $y_j^*(y_i)=\delta_{ij}$. Therefore, for every $j\in\{1,\ldots, n\}$, one has
$$X\ni y_j^*\circ S=\sum_{i=1}^n y_j^*(y_i)x_i^{**}=\sum_{i=1}^n \delta_{ij}x_i^{**}=x_j^{**}.$$
Consequently we get that $S\in X\otimes Y$. Summarising, we have proved that each element of $K_{w^*,w}(X^*, Y)$ can be approximated in norm by an element of $X\otimes Y$, so 
$$K_{w^*,w}(X^*, Y)=X\iten Y$$
and we are done. \end{proof}

As a consequence of Proposition \ref{tensor} and \cite[Theorem 5.2]{gr} we get the following.

\begin{corollary}
Let $M$ be a proper pointed metric space. If either $S(M)$ or $X$ has (AP), then $S(M,X)$ is linearly isometrically isomorphic to $S(M) \widehat{\otimes}_\varepsilon X$.
\end{corollary}

Above corollary as well as basic theory of tensor product spaces give us the key for proving our first duality result in the vector-valued setting.

\begin{theorem}\label{predulibrevectorproper}
Let $M$ be a proper pointed metric space and let $X$ be a Banach one. Assume that $S(M)$ separates points uniformly. If either $\mathcal F(M)$ or $X^*$ has (AP), then
$$S(M,X)^*=\mathcal F(M,X^*).$$
\end{theorem}

\begin{proof}
As $S(M)$ separates points uniformly, $S(M)^*=\mathcal F(M)$ \cite{dal2}. Thus $S(M)$ is an Asplund space. Consequently, we get from above corollary and from \cite[Theorem 5.33]{rya} that 
$$S(M,X)^*=(S(M) \iten X)^*=\mathcal F(M)\widehat{\otimes}_{\pi}X^*=\mathcal F(M,X^*),$$
so we are done.
\end{proof}

Next proposition enlarges the class of metric spaces in which above theorem applies. 

\begin{proposition} \label{holderseppoints}
Let $M$ be a proper metric space. If $(M,\omega \circ d)$ is a proper H\"{o}lder metric space where $\omega$ is a non trivial gauge, then $S(M)$ separates points uniformly. 
\end{proposition}

\begin{proof}
 We will adapt the technique done in \cite[Proposition 3.5]{kalton2004} for the compact case. We will show that, for every $x \neq y \in M$ and every $\ep>0$, there exists $f \in lip(M,d_{\omega})$ such that $|f(x)-f(y)| \geq \dw(x,y)-\ep$ and $\|f\|_{Lip_{\omega}} \leq 1$ (where $\dw$ denotes $\omega \circ d$). 
Fix $\ep >0$ and let $x \neq y \in M$. We denote $a = \dw(x,y)$ and define  $\varphi$: $[0,+\infty[ \to [0,+\infty[$ by the equation:
$$\varphi(t) = \left\{
    \begin{array}{llll}
        t & \mbox{ if } 0 \leq t < a- \ep, \\
        a- \ep & \mbox{ if } a-\ep \leq t < a+ \ep, \\
        -t + 2a & \mbox{ if } a+\ep \leq t < 2a, \\
        0 & \mbox{ if } 2a \leq t. \\
    \end{array}
\right.
$$
Notice that $\|\varphi\|_{Lip} \leq 1$. For every $n \in \mathbb N$ we define a new gauge
$$ \om_n(t) = \inf\{ \om(s) + n (t-s) \, ; \, 0 \leq s \leq t \}.  $$
Note that, for every $t \in [0,+\infty[$, one has that $\om_n(t) \underset{n\to+\infty}{\longrightarrow} \om(t)$.  Finally, for $n \in \mathbb N$,  we consider $h_n$ defined on $M$ by $h_n(z)= \varphi(d_{\om_n}(z,y)) - \varphi(d_{\om_n}(0,y))$. It is straightforward to check that, for $n$ large enough, $|h_n(x)-h_n(y)| = a - \ep = d_{\om}(x,y) - \ep$.  Moreover, given $z$ and $z'$ in $M$, straightforward computations yields the following
\begin{eqnarray*}
|h_n(z)-h_n(z')| &=& |\varphi(d_{\om_n}(z,y)) - \varphi(d_{\om_n}(z',y))| \\
&\leq & \|\varphi \|_{Lip} |d_{\om_n}(z,y) -d_{\om_n}(z',y)| \\
&\leq & d_{\om_n}(z,z').
\end{eqnarray*}
Furthermore, from the definition of $\om_n$, it follows
$$d_{\om_n}(z,z')\leq d_{\om}(z,z')\mbox{ and } d_{\om_n}(z,z')\leq  nd(z,z').$$
Now the first of above inequalities shows that $\|h_n\|_{Lip} \leq 1$  while the second one proves that $h_n \in Lip(M,d) \subset lip(M,d_{\omega})$. It remains to prove that $h_n \in S(M)$. To this aim, fix $\eta >0$, and pick $r>2a+d_{\omega}(0,y)$ such that $\displaystyle{\frac{a}{r-2a-d_{\omega}(0,y)} \leq \eta}$. Now let $z$ and $z'$ be in $M$, and let us discuss by cases:
\begin{itemize}
\item If $z$ and $z'$ are not in $\overline{B}(0,r)$, then $|h_n(z)-h_n(z')|=0 < \eta$.
\item Now suppose that $z \not \in B(0,r)$ and $z' \in B(0,r)$. Now we can still distinguish two more cases:
\begin{itemize}
\item First assume that $d_{\omega}(z',y) \geq 2a$. Then $h_n(z)=h_n(z')=0$ and so $|h_n(z)-h_n(z')|< \eta$ again trivially holds.
\item On the other hand, if $d_{\omega}(z',y) < 2a$, then $|h_n(z')| \leq a$ and so
\begin{eqnarray*}
\frac{|h_n(z) - h_n(z')|}{d_{\om}(z,z')} &\leq& \frac{a}{\dw(z,0)-\dw(z',y)-\dw(0,y)} \\
&\leq& \frac{a}{r-2a-\dw(0,y)} \leq \eta.
\end{eqnarray*}
\end{itemize}
\end{itemize}
This proves that $h_n\in S(M)$ and concludes the proof.\end{proof}

Now we will exhibit some examples of metric and Banach spaces in which Theorem \ref{predulibrevectorproper} applies.

\begin{corollary}\label{coroapcasovectorproper}
Let $M$ be a proper metric space and $X$ be a Banach space. Then $S(M,X)^*=\mathcal F(M,X^*)$ whenever $M$ and $X$ satisfy one of the following assumptions:
\begin{enumerate}
\item $M$ is countable.
\item $M$ is ultrametric.
\item $(M,\omega \circ d)$ is a H\"{o}lder metric space where $\omega$ is a non trivial gauge, and either $\mathcal F(M)$ or $X^*$ has (AP).
\item $M$ is the middle third Cantor set.
\end{enumerate}
\end{corollary}

\begin{proof}
If $M$ satisfies either (1) or (2), then $S(M)$ separates points uniformly and $\mathcal F(M)$ has the approximation property \cite{dal2}. Thus Theorem \ref{predulibrevectorproper} applies.  Moreover, if $M$ satisfies (3) then Proposition \ref{holderseppoints} does the work. Finally, \cite[Proposition 3.2.2]{weaver} yields (4).
\end{proof}

Throughout the rest of the section we will consider a bounded metric space $(M,d)$ and a topology $\tau$ on $M$ such that $(M,\tau)$ is compact and $d$ is $\tau$-lower semicontinuous. We will consider 
\[ lip_\tau(M)=lip(M)\cap \mathcal C(M,\tau),\]
the space of little-Lipschitz functions which are $\tau$-continuous on $M$. Since $M$ is bounded, $lip_\tau(M)$ is a closed subspace of $lip(M)$ and thus it is a Banach space. Moreover, Kalton proved in \cite[Theorem 6.2]{kalton2004} that $lip_\tau(M)^*=\mathcal F(M)$ whenever $M$ is separable and complete and the following condition holds:
\begin{equation}\label{sepaliptau}\tag{P} \forall x,y\in M \ \forall \varepsilon>0 \ \exists f\in B_{lip_\tau(M)} :  |f(x)-f(y)|\geq d(x,y)-\varepsilon.\end{equation}
Recall that above condition holds if, and only if, $lip_\tau(M)$ is $1$-norming for $\mathcal F(M)$ \cite[Proposition 3.4]{kalton2004}. 

Now we can wonder whether there is a natural extension of this result to the vector-valued case. We will prove, following similar ideas to the ones of \cite[Section 5]{gr}, that under suitable assumptions the space
\[ lip_\tau(M,X):= lip(M,X)\cap \{f\colon M\to X : f \text{ is } \tau-||\cdot|| \text{ continuous}\}\]
is a predual of $\mathcal F(M,X^*)$. 

For this, we shall begin by characterising relative compactness in $lip_\tau(M)$.

\begin{lemma}\label{caracompareltau} Let $(M,d)$ be a metric space of radius $R$ and $\tau$ a topology on $M$ such that $(M,\tau)$ is compact and $d$ is $\tau$-lower semicontinuous. Let $\mathcal F$ be a subset of $lip_\tau(M)$. Then $\mathcal F$ is relatively compact in $lip_\tau(M)$ if, and only if, the following three conditions hold:
\begin{enumerate}
\item \label{caracomp1} $\mathcal F$ is bounded.
\item \label{caracomp2}$\mathcal F$ satisfies the following uniform little-Lipschitz condition: for every $\varepsilon>0$ there exists a positive $\delta>0$ such that 
\[ \sup\limits_{0<d(x,y)<\delta}\frac{\vert f(x)-f(y)\vert}{d(x,y)}<\varepsilon\] 
for every $f\in \mathcal F$.
\item \label{caracomp3}$\mathcal F$ is equicontinuous in $\mathcal C((M,\tau))$, i.e. for every $x\in M$ and every $\varepsilon>0$ there exists $U$ a $\tau$-neighbourhood of $x$ such that $y\in U$ implies $\sup\limits_{f\in\mathcal F}\vert f(x)-f(y)\vert<\varepsilon$.
\end{enumerate}\end{lemma}

\begin{proof} In \cite[Theorem 6.2]{kalton2004} it is proved that $lip_\tau(M)$ is isometrically isomorphic to a subspace of a space of continuous functions on a compact set. Indeed, let $K:=\{(x,y,t)\in (M,\tau)\times (M,\tau)\times [0,2R]: d(x,y)\leq t\}$. Then $K$ is compact by $\tau$-lower semicontinuity of $d$. Moreover, the map $\Phi\colon lip_\tau(M)\rightarrow \mathcal C(K)$ defined by 
$$\Phi (f)(x,y,t):=\left\{\begin{array}{cc}
\frac{f(x)-f(y)}{t} & t\neq 0,\\
0 & \mbox{otherwise.}
\end{array}\right.$$
is a linear isometry. Therefore, we have that $\mathcal F$ is relatively compact if, and only if, $\Phi(\mathcal F)$ is relatively compact. By Ascoli-Arzel\`a theorem we get that $\mathcal F$ is relatively compact if, and only if, $\Phi(\mathcal F)$ is bounded and equicontinuous in $\mathcal C(K)$. We will first assume that conditions (\ref{caracomp1}), (\ref{caracomp2}) and (\ref{caracomp3}) hold. It is clear that $\Phi(\mathcal F)$ is bounded, so let us prove equicontinuity of $\Phi(\mathcal F)$. To this aim pick $(x,y,t)\in K$. Now we have two possibilities:
\begin{enumerate}
\item[(i)] If $t\neq 0$ we can find positive number $\eta<t$ such that $t'\in ]t-\eta,t+\eta[$ implies $\left\vert\frac{1}{t}-\frac{1}{t'}\right\vert<\frac{\varepsilon}{4R\alpha}$, where $\alpha=\sup_{f\in\mathcal F}||f||$. Now, as $x$ and $y$ are two points of $M$ and $\mathcal F$ verifies condition (\ref{caracomp3}), we conclude the existence of $U$ a $\tau$-neighbourhood of $x$ and $V$ a $\tau$-neighbourhood of $y$ in $M$ verifying $x'\in U, y'\in V$ implies $\vert f(x)-f(x')\vert+\vert f(y)-f(y')\vert<\frac{\varepsilon t }{2}$ for every $f\in\mathcal F$. Now, given $(x',y',t')\in (U\times V\times ]t-\eta,t+\eta[)\cap K$, one has
$$\vert \Phi f(x,y,t)-\Phi f(x',y',t')\vert=\left\vert\frac{f(x)-f(y)}{t}-\frac{f(x')-f(y')}{t'} \right\vert\leq$$
$$\left\vert\frac{1}{t}-\frac{1}{t'} \right\vert |f(x')-f(y')|+\frac{1}{t}|f(x)-f(x')+f(y)-f(y')|$$
$$\leq\frac{\varepsilon}{4R\alpha} ||f||d(x',y')+\frac{\varepsilon t }{2t} \leq \frac{\varepsilon}{2}+\frac{\varepsilon}{2}=\varepsilon$$
for every $f\in\mathcal F$, which proves equicontinuity of $\Phi(f)$ at $(x,y,t)$.
\item[(ii)] 
If $t=0$ then $x=y$. Pick an arbitrary $\varepsilon>0$. By (\ref{caracomp2}) we get a positive $\delta$ such that $0<d(x,y)<\delta$ implies $\frac{\vert f(x)-f(y)\vert}{d(x,y)}<\varepsilon$ for every $f\in\mathcal F$. Now, given $(x',y',t)\in (M\times M\times [0,\delta[)\cap K$ we have $d(x',y')\leq t<\delta$ and so, given $f\in\mathcal F$, it follows
$$\vert \Phi f(x',y',t)\vert\leq \frac{\vert f(x')-f(y')\vert}{t}<\varepsilon\frac{d(x',y')}{t}\leq \varepsilon,$$
which proves equicontinuity at $(x,x,0)$. \end{enumerate}
Both previous cases prove that $\Phi(\mathcal F)$ is equicontinuous whenever conditions (\ref{caracomp1}), (\ref{caracomp2}) and (\ref{caracomp3}) are satisfied. 

Conversely, assume that $\Phi(\mathcal F)$ is equicontinuous in $\mathcal C(K)$. It is clear that $\mathcal F$ is bounded, so let us prove that conditions (\ref{caracomp2}) and (\ref{caracomp3}) are satisfied. We shall begin by proving (\ref{caracomp3}), for which we fix $x\in M$ and $\varepsilon>0$. Given $t\in[0,2R]$, by equicontinuity of $\Phi(\mathcal F)$ at the point $(x,x,t)$, we can find $U_t$ a $\tau$-neighbourhood of $x$ and $\eta_t>0$ such that $x'\in U_t$ and $t'\in ]t-\eta_t, t+\eta_t[$ implies $\vert \Phi f(x,x',t')\vert<\frac{\varepsilon}{2R}$ for every $f\in \mathcal F$. Then $[0,2R] \subset \bigcup_t ]t-\eta_t, t+\eta_t[$ and thus there exist $t_1,\ldots, t_n$ such that $[0,2R] \subset \bigcup_{i=1}^n ]t_i-\eta_{t_i}, t_i+\eta_{t_i}[$. Now take $U=\bigcap_{i=1
}^n U_{t_i}$. We will show that $U$ is the desired $\tau$-neighbourhood of $x$. Pick $x'\in U$. Then there exists $t_i$ such that $d(x,x')\in  ]t_i-\eta_{t_i}, t_i+\eta_{t_i}[$. Since $x'\in U_{t_i}$ we get
$$\vert \Phi f(x,x',d(x,x'))\vert = \left\vert \frac{f(x)-f(x')}{d(x,x')}\right\vert<\frac{\varepsilon}{2R}$$
and thus $|f(x)-f(x')|<\varepsilon$ for every $x'\in U$ and $f\in \mathcal F$. This proves that $\mathcal F$ is equicontinuous at every $x\in M$.

Finally, let us prove condition (\ref{caracomp2}). To this aim pick a positive $\varepsilon$. For every $x\in M$ we have, from equicontinuity of $\Phi(\mathcal F)$ at $(x,x,0)$, the existence of $U_x$ a $\tau$- open neighbourhood of $x$ in $M$ and a positive $\delta_x>0$ such that $x',y'\in U_x$ and $0<t<\delta_x$ implies $\vert \Phi f(x',y',t)\vert<\varepsilon$ for every $f\in\mathcal F$. 

As $M\times M=\bigcup\limits_{x\in M} U_x\times U_x$, we get by compactness the existence of $x_1,\ldots, x_n\in M$ such that $M\times M\subseteq \bigcup\limits_{i=1}^n U_{x_i}\times U_{x_i}$. Pick $\delta:=\min\limits_{1\leq i\leq n}\delta_{x_i}$. Now, if $x,y\in M$ verifies that $0<d(x,y)<\delta$ then there exists $i\in\{1,\ldots, n\}$ such that $x,y\in U_{x_i}$. As $d(x,y)<\delta\leq \delta_{x_i}$ we get 
$$\frac{\vert f(x)-f(y)\vert}{d(x,y)}=\vert \Phi f(x,y,d(x,y))\vert<\varepsilon$$
for every $f\in\mathcal F$, which proves (\ref{caracomp2}) and finishes the proof.\end{proof}

Previous lemma allows us to identity $lip_\tau(M,X)$ as a space of compact operators from $X^*$ to $lip_\tau(M)$.

\begin{theorem} \label{isomcompop}
Let $M$ be a pointed metric space and let $\tau$ be a topology on $M$ such that $(M,\tau)$ is compact and $d$ is $\tau$-lower semicontinuous. Then  $lip_\tau(M,X)$ is isometrically isomorphic to $K_{w^*,w}(X^*,lip_\tau(M))$. Moreover, if either $lip_\tau(M)$ or $X$ has (AP), then
$lip_\tau(M,X)$ is isometrically isomorphic to $lip_\tau(M) \widehat{\otimes}_\varepsilon X$.
\end{theorem}

\begin{proof} It is shown in \cite{jsv} that $f\mapsto f^t$ defines an isometry from $Lip(M,X)$ onto $L_{w^*,w^*}(X^*, Lip(M)))$, where $f^t(x^*)=x^*\circ f$. Let $f$ be in $lip_\tau(M,X)$ and let us prove that $f^t\in K_{w^*,w}(X^*,lip_\tau(M))$. Notice that  $x^*\circ f$ is $\tau$-continuous for every $x^*\in X^*$. Moreover, for every $x\neq y\in M$ and every $x^*\in X^*$, we have
\begin{equation}\label{uniflittle} \frac{|x^*\circ f(x)-x^*\circ f(y)|}{d(x,y)} \leq ||x^*|| \frac{\Vert f(x)-f(y)\Vert}{d(x,y)}
\end{equation}
thus $x^\ast\circ f \in lip(M)$. Therefore $f^t(X^*)\subset lip_\tau(M)$. We claim that $f^t(B_{X^*})$ is relatively compact in $lip_\tau(M)$. In order to show that, we need to check the conditions in Lemma \ref{caracompareltau}. First, it is clear that $f^t(B_{X^*})$ is bounded. Moreover, it follows from (\ref{uniflittle}) that the functions in $f^t(B_{X^*})$ satisfy the uniform little-Lipschitz condition. Finally, $f^t(B_{X^*})$ is equicontinuous in the sense of Lemma \ref{caracompareltau}. Indeed, given $x\in M$ and $\varepsilon>0$, there exists a $\tau$-neighbourhood $U$ of $x$ such that $||f(x)-f(y)||<\varepsilon$ whenever $y\in M$. That is, 
\[\sup_{x^*\in B_{X^*}} |x^*\circ f(x)-x^*\circ f(y)|<\varepsilon\]
whenever $y\in U$, as we wanted. Now, Lemma \ref{caracompareltau} implies that $f^t(B_{X^*})$ is a relatively compact subset of $lip_\tau(M)$ and thus $f^t\in K(X^*, lip_\tau(M))\cap L_{w^*,w^*}(X^*, Lip(M))$. Finally, the set $\overline{f^t(B_{X^*})}$ is norm-compact and thus every coarser Hausdorff topology agrees on it with the norm topology. In particular, the weak topology of $lip_\tau(M)$ agrees on $f^t(B_{X^*})$ with the inherited weak-star topology of $Lip(M)$. Thus $f^t|_{ B_{X^*}}\colon B_{X^*} \to lip_\tau(M)$ is $w^*-w$ continuous. By \cite[Proposition 3.1]{Kim} we have that $f^t\in K_{w^*,w}(X^*,lip_\tau(M))$. 

It only remains to prove that the isometry is onto. For this take $T\in K_{w^*,w}(X^*,lip_\tau(M))$. We claim that $T$ is $w^*-w^*$ continuous from $X^*$ to $Lip(M)$. Indeed, assume that $\{x_\alpha^*\}$ is a net in $X^*$ weak-star convergent to some $x^*\in X^*$. Since every $\gamma\in \mathcal F(M)$ is also an element in $lip_\tau(M)^*$, we get that $\langle\gamma, Tx^*_\alpha\rangle$ converges to $\langle\gamma, Tx^*\rangle$. Thus, $T\in L_{w^*,w^*}(X^*, Lip(M)))$. By the isometry described above, there exists $f\in Lip(M,X)$ such that $T=f^t$. Let us prove that $f$ actually belongs to $lip_\tau(M,X)$. As $f^t(B_{X^*})$ is relatively compact, then by Lemma \ref{caracompareltau} we have that for every $\varepsilon>0$ there exists and $\delta>0$ such that 
\[ \sup_{0<d(x,y)<\delta} \frac{| x^*\circ f(x)-x^*\circ f(y)|}{d(x,y)} < \varepsilon\]
for each $x^*\in B_{X^*}$. By taking supremum with $x^*\in B_{X^*}$ we get that 
\[ \sup_{0<d(x,y)<\delta} \frac{\Vert f(x)-f(y)\Vert}{d(x,y)} \leq \varepsilon,\]
so $f\in lip(M, X)$. We will prove, to finish the proof, that $f$ is $\tau-||\cdot||$ continuous. To this aim pick $y\in M$ and $\varepsilon>0$. By equicontinuity of $f^t(B_{X^*})$ we can find $U$ a $\tau$-neighbourhood of $y$ such that $|x^*\circ f(y')-x^*\circ f(y)|<\varepsilon$ for every $x^*\in B_{X^*}$ and $y'\in U$. Now, 
\[ ||f(y')-f(y)||=\sup_{x^*\in B_{X^*}}|x^*(f(y')-f(y))| \leq \varepsilon\]
for every $y'\in U$. Consequently, $f$ is $\tau-||\cdot||$ continuous. So $f\in lip_\tau(M,X)$, as desired. 

Finally, if either $lip_\tau(M)$ or $X$ has the approximation property, then Proposition \ref{tensor} yields the equality $K_{w^*,w}(X^*,lip_\tau(M))=lip_\tau(M) \widehat{\otimes}_\varepsilon X$.\end{proof}

Now we get our second duality result for vector-valued Lipschitz-free Banach spaces, which extends \cite[Theorem 6.2]{kalton2004}.

\begin{theorem} \label{tauduality}
Let $M$ be a separable complete bounded pointed metric space. Suppose that $\tau$ is a metrizable topology on
$M$ so that $(M,\tau)$ is compact satisfying the property (\ref{sepaliptau}). If either $\mathcal F(M)$ or $X^*$ has (AP), then $lip_\tau(M,X)^*=\mathcal F(M,X^*).$
\end{theorem}

\begin{proof}
By \cite[Theorem 6.2]{kalton2004} we have that $lip_\tau(M)$ is a predual of $\mathcal F(M)$. Consequently, $\mathcal F(M)$ has (RNP). Therefore, we get from Theorem \ref{isomcompop} and from \cite[Theorem 5.33]{rya} that 
$$lip_\tau(M,X)^*=(lip_\tau(M)\iten X)^*=\mathcal F(M)\widehat{\otimes}_{\pi}X^*=\mathcal F(M,X^*),$$
which finishes the proof.\end{proof}

Last result applies to the following particular case (see Proposition 6.3 in \cite{kalton2004}). Given two Banach spaces $X,Y$, and $\omega$ a non trivial gauge, we will denote $lip_{\om,*}(B_{X^*},Y):= lip_{w^*}((B_{X^*}, \omega\circ ||\cdot||),Y)$.

\begin{corollary} Let $X$ and $Y$ be Banach spaces, and let $\omega$ be a non trivial gauge. Assume that $X^*$ is separable and that either $\F(B_{X^*},\omega\circ ||\cdot||)$ or $Y^*$ has (AP). Then $lip_{\om,*}(B_{X^*},Y)$ is linearly isometrically isomorphic to $lip_{\om,*}(B_{X^*}) \iten Y $ and $lip_{\om,*}(B_{X^*},Y)^*=\F((B_{X^*},\om\circ ||\cdot||),Y^*)$. 
\end{corollary}

Finally, we will take advantage of the notion of unconditional almost squareness, introduced in \cite{gr}, in order to prove non duality of the space $lip_{\om,*}(B_{X^*},Y)$ under above hypotheses. According to \cite{gr}, a Banach space $X$ is said to be \textit{unconditionally almost square} (UASQ) if, for each $\varepsilon>0$, there exists a subset $\{x_\gamma\}_{\gamma\in \Gamma}\subseteq S_X$ such that
\begin{enumerate}
\item\label{defi211} For each $\{y_1,\ldots, y_k\}\subseteq S_X$ and $\delta>0$ the set
$$\{\gamma\in\Gamma :  \Vert y_i\pm x_\gamma\Vert\leq 1+\delta\ \forall i\in\{1,\ldots, k\}\}$$
is non-empty.
\item\label{defi212} For every $F$ finite subset of $\Gamma$ and every choice of signs $\xi_\gamma\in \{-1,1\}$, $\gamma\in F$, it follows $\Vert\sum_{\gamma\in F} \xi_\gamma x_\gamma\Vert\leq 1+\varepsilon$.
\end{enumerate}
It is known that there is not any dual UASQ Banach space \cite[Theorem 2.5]{gr}.

\begin{proposition}\label{UASQparanodual} Let $X$ and $Y$ be Banach spaces, and let $\omega$ be a non trivial gauge. Assume that $X^*$ is separable. Then $lip_{\om,*}(B_{X^*},Y)$ is UASQ. In particular, it is not isometric to any dual Banach space. 
\end{proposition}

\begin{proof} First we prove that $lip_{\om,*}(B_{X^*})$ is UASQ. By \cite[Proposition 3.3]{gr}, it suffices to show that there exists a point $x_0^*\in B_{X^*}$ and sequences $r_n$ of positive numbers and $f_n\in lip_{\om,*}(B_{X^*})$ such that $f_n\neq 0$, $f_n(x_0^*)=0$ and $f_n$ vanishes out of $B(x_0^*, r_n)$. Let $x_0^*\in S_{X^*}$ be a continuity point of the identity $I\colon (B_{X^*}, w^*)\to (B_{X^*},\Vert\cdot\Vert)$, that is, $x_0^*$ has relative weak-star neighbourhoods of arbitrarily small diameter. Take a sequence $\{W_n\}$ of relative weak-star neighbourhoods of $x_0^*$ and a sequence $r_n\rightarrow 0$ such that $0\notin W_{n}\subset B(x_0^*,r_n)\subset W_{n-1}$. For each $n$ choose $x_n^*\in W_{n}\setminus\{x_0^*\}$ and define $A_n=\{x_0^*,x_n^*\}\cup (B_{X^*}\setminus W_{n})$. Consider $f_n\colon A_n \to \mathbb R$ given by $f_n(x_n^*)=1$ and $f_n(x)=0$ otherwise. Then $A_n$ is weak-star closed and $f_n\in Lip(A_n,\Vert\cdot\Vert)\cap \mathcal C(A_n, w^*)$. By \cite[Corollary 2.5]{mat}, there exists $g_n\in Lip(B_{X^*},\Vert\cdot\Vert)\cap \mathcal C(B_{X^*}, w^*)$ extending $f_n$. Then $g_n$ is a non zero Lipschitz function which is weak-star continuous and  vanishes on $B_{X^*}\setminus B(x_0^*,r_n)$. Finally notice that $Lip(B_{X^*},\Vert\cdot\Vert)\subset lip_\omega(B_{X^*})$. Thus $\{g_n\}\subset lip_{\omega,*}(B_{X^*})$ and so $lip_{\om,*}(B_{X^*})$ is UASQ.

Now, by Theorem \ref{isomcompop}, we have that $lip_{\om,*}(B_{X^*}, Y)$ is a subspace of \linebreak $K(Y^*, lip_{\om,*}(B_{X^*}))$ which clearly contains $lip_{\om,*}(B_{X^*})\otimes Y$. Proposition 2.7 in \cite{gr} provides unconditional almost squareness of $lip_{\om,*}(B_{X^*},Y)$. Finally, the non-duality of this space follows from Theorem 2.5 in \cite{gr}. 
\end{proof}

\begin{remark}
\begin{enumerate}
\item Notice that previous result can be strengthened in case of being $Y$ a separable space with (AP). In fact, in that case $lip_{\om,*}(B_{X^*},Y)=lip_{\om,*}(B_{X^*}) \iten Y$ is a separable Banach space which contains an isomorphic copy of $c_0$ \cite[Lemma 2.6]{all}, so it can not be even isomorphic to any dual Banach space. Moreover,  up the best of our knowledge, the fact that $lip_{\omega,*}(B_{X^*},Y)$ is not a dual Banach space was not known even in real case.

\item Previous result has an immediate consequence in terms of octahedrality in Lipschitz-free Banach spaces. Recall that a Banach space $X$ is said to have an \textit{octahedral norm} if for every finite-dimensional subspace $Y$ and for every $\varepsilon>0$ there exists $x\in S_X$ verifying that $\Vert y+\lambda x\Vert>(1-\varepsilon)(\Vert y\Vert+\vert \lambda\vert)$ for every $y\in Y$ and $\lambda\in\mathbb R$. Notice that, given Banach spaces $X$ and $Y$ under the assumption of Proposition \ref{UASQparanodual}, it follows that $\mathcal F((B_{X^*},\omega\circ \Vert\cdot\Vert), Y^*)=\mathcal F((B_{X^*},\omega\circ \Vert\cdot\Vert))\pten Y^*$ has an octahedral norm because of \cite[Corollary 2.9]{llr}. Notice that this gives a partially positive answer to \cite[Question 2]{blr}, where it is wondered whether octahedrality in vector-valued Lipschitz-free Banach spaces actually relies on the scalar case.
\end{enumerate} 
\end{remark}

\section{Schur property on vector-valued Lipschitz-free spaces} \label{schur}
\bigskip

According to \cite{gjl}, a Banach space $X$ is said to have the \textit{Schur property} whenever every weakly null sequence is actually a norm null sequence, and $X$ is said to have the \textit{Dunford-Pettis property} whenever every weakly compact operator from $X$ into a Banach space $Y$ is \textit{completely continuous}, i.e. carries weakly compact sets into norm compact sets.

It is known that a dual Banach space $X^*$ has the Schur property if, and only if, $X$ has the Dunford-Pettis property and does not contain any isomorphic copy of $\ell_1$ \cite[Theorem 5.2]{gjl}. So, in order to analyse the Schur property in $\mathcal F(M,X^*)$, it can be useful analysing the Dunford-Pettis property in the predual in case such a predual exists. For this, in proper case, we can go much further.

\begin{theorem}\label{hDPPcasoreal}
Let $M$ be a proper metric space such that $S(M)$ separates points uniformly. Then $S(M)$ does not contain any isomorphic copy of $\ell_1$ and has the hereditary Dunford-Pettis property, i.e. every closed subspace of $S(M)$ has the Dunford-Pettis property. 
\end{theorem}

\begin{proof} By \cite{dal2} we get that $S(M)$ is $(1+\varepsilon)$ isometric to a subspace of $c_0$, which is known to have the hereditary Dunford-Pettis (see e.g. \cite{cem}). Consequently, $S(M)$ has the hereditary Dunford-Pettis property. Obviously, previous condition also implies that $S(M)$ does not contain any isomorphic copy of $\ell_1$.
\end{proof}

Above theorem not only applies in the scalar valued version of $S(M)$ but also in the vector valued one. Indeed, we get the following result.

\begin{theorem}
Let $M$ be a proper metric space such that $S(M)$ separates points uniformly. Assume that $X$ is a Banach space which the hereditary Dunford-Pettis property and that $X$ does not contain any isomorphic copy of $\ell_1$. If either $X$ or $S(M)$ has (AP), then $S(M,X)$ does not contain any isomorphic copy of $\ell_1$ and has the hereditary Dunford-Pettis property. 
\end{theorem}

\begin{proof}
As $S(M)\iten X=S(M,X)$ holds because of Theorem \ref{predulibrevectorproper}, then it does not contain any isomorphic copy of $\ell_1$ \cite[Corollary 4]{ros}. Moreover, as $S(M)$ is isomorphic to a subspace of $c_0$, then $S(M,X)=S(M)\iten X$ is isomorphic to a subspace of $c_0\iten X=c_0(X)$. As $c_0(X)$ has the hereditary Dunford-Pettis property whenever $X$ has the hereditary Dunford-Pettis property \cite[Theorem 3.1]{ko} we get that $S(M,X)$ has the hereditary Dunford-Pettis property, so we are done.\end{proof}

We say that a Banach space $X$ has the \textit{strong Schur property} if there exists a constant $K > 0$ such
that, given $\delta >0$ and a sequence $(x_n)_{n\in\mathbb N}$ in the unit ball of $X$ verifying $\Vert x_n-x_m\Vert\geq \delta$ for all $n\neq m$, then $\{x_n\}$ contains a subsequence that is $\frac{K}{\delta}$ -equivalent to the unit vector basis of $\ell_1$. As it is known that a dual Banach space $X^*$ has the strong Schur property whenever $X$ does not contain any isomorphic copy of $\ell_1$ and has the hereditary Dunford-Pettis property \cite[Theorem 4.1]{ko}, we get the following Corollary.

\begin{corollary}\label{corostrongschur}
Let $M$ be a proper metric space such that $S(M)$ separates points uniformly. Assume that $X$ is a Banach space with the hereditary Dunford-Pettis property and that $X$ does not contain any isomorphic copy of $\ell_1$. If either $X^*$ or $\mathcal F(M)$ has (AP), then $\mathcal F(M,X^*)$ has the strong Schur property.
\end{corollary}

\begin{remark}
Above corollary should be compared with \cite[Proposition 2.11]{pet} in real case.
\end{remark}

Bearing in mind the identification $\mathcal F(M,X)=\mathcal F(M)\pten X$, philosophically we can say that we have obtained a result about Schur property in $\mathcal F(M,X)$ from tensor product theory. Now we are going to state a result in the reverse direction, i.e. we find conditions which guarantee that $\mathcal F(M,X)$ has the Schur property and, as a consequence, we get examples of Banach spaces with Schur property whose projective tensor product still has the Schur property. Note that such examples are interesting because, up the best of our knowledge, it is an open problem how projective tensor product preserves previous property \cite[Remark 6]{gg}. 

For this, we will analyse the uniformly discrete case, for which Kalton proved in \cite{kalton2004} that the scalar valued Lipschitz-free Banach space has the Schur property. Here we extend this result to a vector-valued setting.

\begin{proposition}
Let $(M,d)$ be an uniformly discrete metric space, that is, assume that $\theta = \inf_{m_1 \neq m_2} d(m_1,m_2) > 0$, and let $X$ be a Banach space with the Schur property. Then $\F(M,X)$ has the Schur property.
\end{proposition}

\begin{proof}
For this purpose we will need Kalton's decomposition (see Lemma 4.2 in \cite{kalton2004}). That is, there exist a constant $C>0$ and a sequence of operators $T_k\colon \F(M) \to \F(M_k)$, where $k\in \mathbb Z$ and $M_k$ denotes the closed ball $\overline{B}(0,2^k)$, satisfying 
\[ \gamma = \sum_{k\in \mathbb Z} T_k \gamma \text{ unconditionally and } \sum_{k\in \mathbb Z} \|T_k \gamma \| \leq C\|\gamma\|\]
for every $\gamma\in \mathcal F(M)$. Now, using this decomposition, we can consider $S\colon \F(M) \to (\sum\F(M_k))_{\ell_1}$ defined by $S \gamma = (T_k \gamma)_k$. So $S$ defines an isomorphism between $\F(M)$ and a closed subspace of $(\sum\F(M_k))_{\ell_1}$.

We will show that the image of $\F(M)$ is complemented in $(\sum\F(M_k))_{\ell_1}$. To achieve this we define $P$: $(\sum\F(M_k))_{\ell_1} \to S(\F(M))$ by $P((\gamma_k)_k) = \nolinebreak (T_k \gamma)_k$, where $\gamma = \sum_k \gamma_k$. Then $P$ is a well defined projection. Indeed, if \linebreak $(\gamma_k)_k \in (\sum\F(M_k))_{\ell_1}$ then $P(P((\gamma_k)_k)) = P((T_k\gamma)_k)$. Now, if we define \linebreak $\gamma := \sum_{k\in \mathbb Z} T_k \gamma$, it follows that $P((T_k\gamma)_k) = (T_k \gamma)_k$, which proves that $P\circ P=P$ . Notice that $P$ is continuous since, given $(\gamma_k)\in (\sum\F(M_k))_{\ell_1}$, if we define $\gamma:=\sum_{k\in\mathbb Z} \gamma_k$, we have the following chain of inequalities 
\[\Vert P((\gamma_k))\Vert =\left \|\sum_{k\in \mathbb Z} T_k \gamma \right\| \leq \sum_{k\in \mathbb Z} \|T_k \gamma \| \leq C \|\gamma\| = C \left \|\sum_{k\in \mathbb Z} \gamma_k\right \| \leq C \sum_{k\in \mathbb Z} \|\gamma_k\|.\]

Thus $\F(M) \pten X$ is isomorphic to a subspace of $(\sum\F(M_k))_{\ell_1} \pten X$  \cite[Proposition 2.4]{rya}. It is not difficult to prove that $(\sum\F(M_k))_{\ell_1} \pten X$ is isometrically isomorphic to $(\sum \F(M_k) \pten X)_{\ell_1}$. Consequently, we have that $\F(M,X)$ is isomorphic to a subspace of $(\sum \F(M_k,X))_{\ell_1}$.

In order to finish the proof, we will prove that $\F(M_k,X)$ has the Schur property for every $k$, which will be enough since the Schur property is stable under $\ell_1$ sums \cite{tan} and by passing to subspaces. To do that, we will show that $\F(M_k,X)$ is isomorphic to $\ell_1(M_k, X)$ (the space of all absolutely summable families in $X$ indexed by $M_k$), which enjoys the Schur property since $X$ has it. Consider $F$ a finite set, $(a_i)_{i \in F}$ a finite sequence of scalars and $\gamma = \sum_{i \in F} a_i \delta_{m_i,y_i} \in \F(M_k,X)$. Using the triangle inequality we have $\|\gamma\| \leq \sum_{i \in F} |a_i| \|\delta_{m_i}\| \|y_i\| \leq 2^k\sum_{i \in F} |a_i| \|y_i\|$. Moreover, for each $i\in F$, pick $x_i^* \in X^*$ such that $x_i^*(y_i)=sign(a_i)\|x_i\|$ and define $f$: $M_k \to X^*$ by the equation
$$f(m):=\begin{cases}
x_i^* & \mbox{if }m=m_i  \mbox{ for some } i\in F,\\
0 & \mbox{otherwise.}
\end{cases}$$ 
Since $2^{-k} \|f\|_{\infty} \leq \|f\|_{Lip} \leq 2\theta^{-1}\|f\|_{\infty}$, we get that $\|f\|_{Lip} \leq 2\theta^{-1}$. Thus $\|\gamma\| \geq \langle \frac{\theta}{2}f ,  \gamma  \rangle = \frac{\theta}{2}\sum_{i \in F} |a_i|\|x_i\|$. This proves that the linear operator \linebreak $T$: $\F(M_k,X) \to \ell_1(M_k, X)$ defined by $T(\sum_{i \in F} a_i \delta_{m_i,x_i}) = (z_m)_{m \in M_k}$, where $z_{m_i} = a_i x_i$ and $z_m=0$ otherwise, is an isomorphism.
\end{proof}

\begin{remark}\label{unifdiscstrongSchur}
Since $\mathcal F(M_k,X^*)$ is isomorphic to $\ell_1(X^*)$, we get that $\mathcal F(M,X^*)$ has the strong Schur whenever $X^*$ has it in the above proposition. Indeed, this follows from the two next propositions.
\end{remark}

\begin{proposition} \label{sumfinite}
Let $(X_k)_{k=1}^N$ be a finite family of Banach spaces. Assume that each $X_k$ has the strong Schur property with the same constant $K$ in the definition of this property. Then $X=(\sum_{k=1}^N X_k)_{\ell_1}$ has the strong Schur property with constant $K+\varepsilon$ for every $\varepsilon >0$. 
\end{proposition}

\begin{proof}
Fix $\delta >0$ and let $(x_n)_n$ be a $\delta$-separated sequence in the unit ball of $X$. We denote $(x_n')_n \prec (x_n)_n$ to mean that $(x_n')_n$ is a subsequence of $(x_n)_n$. We consider the following quantity:
\begin{eqnarray*}
\Delta = \sup \left\{\sum_{k=1}^N \delta_k : \exists (x_n')_n \prec (x_n)_n , \forall k=1\cdots N, (x_n'(k))_n \mbox{ is }\delta_k- separated \right\}
\end{eqnarray*}
We will show that $\Delta \geq \delta$. Let $\ep>0$ arbitrary, and assume that $\Delta<\delta -\ep$. Then there exist $(x_n')_n$ a subsequence of $(x_n)_n$ and $(\delta_1,\cdots,\delta_N) \in [0,2]^N$ such that, for every $k$, $(x_n'(k))_n$ is $\delta_k$-separated and $\sum_{k=1}^N \delta_k > \Delta -\ep$. By definition of $\Delta$, there only exists a finite number of couples $(n,m)\in \mathbb N^2$ with $n\neq m$ satisfying 
$$\|x_n'(k) - x_m'(k)\|> \delta_k + \frac{\ep}{N}\ \mbox{for every }k\in\mathbb N.$$
By passing to a further subsequence we can assume that, for every $n \neq m$ and every $k$, $\|x_n(k)-x_m(k)\| \leq \delta_k +\varepsilon/N$. Thus $\sum_{k=1}^N \|x_n(k) - x_m(k) \| \leq \sum_{k=1}^N \delta_k + \varepsilon < \delta$, which contradicts the $\delta$-separation of the original sequence. Consequently $\Delta \geq \delta - \ep$. Since $\varepsilon$ was arbitrary, we get that $\Delta \geq \delta$.

We now consider $(x_n')_n$ a subsequence of $(x_n)_n$ such that, for every $k$, $(x_n'(k))_n$ is $\delta_k$-separated and $\sum_{k=1}^N \delta_k \geq \delta -\varepsilon$. Since each $X_k$ has the strong Schur property, for every $k$ with $\delta_k>0$, there exists a subsequence of $(x_n'(k))_n$, still denoted by $(x_n'(k))_n$ for convenience, such that $(x_n'(k))_n$ is $(K/\delta_k)$-equivalent to the $\ell_1$-basis. Next, by a diagonal argument, there exists $(x_n'')_n$ a subsequence of $(x_n')_n$ such that, for every $k$ with $\delta_k>0$, $(x_n''(k))_n$ is $(K/\delta_k)$-equivalent to the $\ell_1$-basis. Then a simple computation shows that, for every $(a_i)_{i=1}^m \in \R^m$, it follows
\begin{eqnarray*}
\left\|\sum_{i=1}^m a_i x_i''\right \| &=& \sum_{k=1}^N \left\|\sum_{i=1}^m a_i x_i''(k)\right \| \\
&\geq& \sum_{k=1}^N \frac{\delta_k}{K} \sum_{i=1}^m |a_i| \\
&\geq& \frac{\delta - \varepsilon}{K} \sum_{i=1}^m |a_i|.
\end{eqnarray*}
This proves that $(x_n'')_n$ is $\frac{K}{\delta-\varepsilon}$-equivalent to the $\ell_1$-basis.
\end{proof}

Now extend the previous result to infinite $\ell_1$-sums. To achieve this we need to assume that the spaces $X_k$ are dual ones. 

\begin{proposition}
Let $(X_k)_{k\in \mathbb N}$ be a family of Banach spaces. Assume that each $X_k^*$ has the strong Schur property with the same constant $K$ in the definition of this property. We consider $X=(\sum_{k\in \mathbb N} X_k)_{c_0}$ and its dual space $X^* = (\sum_{k\in \mathbb N} X_k^*)_{\ell_1}$. Then $X^*$ has the strong Schur property with constant $\max\{2K+\varepsilon,4+\varepsilon\}$ for every $\varepsilon>0$. 
\end{proposition}

\begin{proof}
For $N\in \mathbb N$, we denote $P_N$: $X^* \to (\sum_{k=1}^N X_k^*)_{\ell_1}$ the norm-one projection on the $N$ first coordinates. Fix $\delta >0$ and let $(x_n)_n$ be a $\delta$-separated sequence in the unit ball of $X^*$. Fix also $\varepsilon>0$. Now two cases may occur.

\textbf{First case.} There exist $N\in \mathbb N$ such that there is $(x_n')_n$ a subsequence of $(x_n)_n$ satisfying $d(x_n',\sum_{k=1}^N X_k^*) \leq \delta/4$. Then a straightforward computation using the triangle inequality shows that $(P_N(x_n'))_n$ is $(\delta/2)$-separated. Thus, according to Proposition \ref{sumfinite}, $(P_N(x_n'))_n$ admit a subsequence $(\frac{2K+\varepsilon}{\delta})$-equivalent to the $\ell_1$-basis. For convenience we still denote the same way the subsequence considered. Now consider $(a_i)_{i=1}^m \in \R^m$, and let us estimate the following norm
\begin{eqnarray*}
\left\|\sum_{i=1}^m a_i x_i' \right\| &\geq& \left \|\sum_{i=1}^m a_i P_N(x_i') \right\| \geq \frac{\delta}{2K+\varepsilon} \sum_{i=1}^m |a_i|.
\end{eqnarray*}
This ends the first case.
 
\textbf{Second case.} For every $N \in \mathbb N$ and every subsequence $(x_n')_n$, there exist $n$ such that $d(x_n',\sum_{k=1}^N X_k^*)>\frac{\delta}{4}$. Passing to a subsequence and using \cite[Lemma 2.13]{pet} we can assume that $(x_n)_n$ is $w^*$ convergent to $0$ and that $\|x_n\| \geq \frac{\delta}{2}$ for every $n$. We will construct by induction a subsequence with the desired property. To achieve this, fix $(\varepsilon_i)_i$ a sequence of positives real numbers smaller than $\frac{\delta}{4}$ such that $\prod_{i=1}^{+\infty} (1-\varepsilon_i) \geq 1-\varepsilon$ and take $C:=4\sum_{k=1}^{+\infty} \frac{\varepsilon_k}{\delta}< \ep$. We begin with the construction of a sequence in $X$ very close to $(x_n)_n$ which is equivalent to the $\ell_1$-basis, and after this we will deduce what we want from the principle of small perturbations (see for example \cite[Theorem 1.3.9]{albiac}). More precisely we will construct a sequence $(P_{K_i}(x_{n_i}))_i$ which is $\frac{4}{\delta(1-\ep)}$-equivalent to the $\ell_1$-basis and such that $\|P_{K_i}(x_{n_i})-x_{n_i}\| \leq \varepsilon_i$. 

\textit{First of all}, we set $n_1=1$ and $N_1 \in \mathbb N$ such that $\|P_{N_1}x_{n_1}\| \geq \|x_{n_1}\| - \varepsilon_1$.

\textit{Construction of $n_2>n_1$.} Since $P_{N_1}$ is $w^*$-continuous, $(P_{N_1}(x_n))_n$ is $w^*$-null. We apply \cite[Lemma 1.5.1]{albiac}, so there exists $m>n_1$ such that for all $n\geq m$ and for all $(\lambda_1,\lambda_2)\in \R^2$, $$\|\lambda_1 P_{N_1}(x_{n_1}) + \lambda_2 P_{N_1}(x_n) \| \geq (1- \varepsilon_2) \|\lambda_1P_{N_1}(x_{n_1})\|.$$
Now, using the assumption of the second case, there exist $n_2\geq m$ such that $\|x_{n_2} - P_{K_1}(x_{n_2})\| > \frac{\delta}{4}$. We then pick $N_2 > N_1$ such that \linebreak $\|P_{N_2}(x_{n_2}) - P_{N_1}(x_{n_2}) \| > \frac{\delta}{4}$ and $\|P_{N_2}(x_{n_2})-x_{n_2}\| < \varepsilon_{2}$. Next the following inequalities hold:
\begin{eqnarray*}
\left\|\lambda_1 P_{N_1}(x_{n_1}) + \lambda_2 P_{N_2}(x_{n_2})\right \| &=& \|\lambda_1 P_{N_1}(x_{n_1}) + \lambda_2 P_{N_1}(x_{n_2})\| \\
&+& \|\lambda_2 [P_{N_2}-P_{N_1}](x_{n_2})\| \\
&>& (1-\varepsilon_2) \|\lambda_1 P_{N_1}(x_{n_1})\| + |\lambda_2|\frac{\delta}{4} \\
&>& (1-\varepsilon_1)(1-\varepsilon_2)\frac{\delta}{4}|\lambda_1|+ |\lambda_2|\frac{\delta}{4} \\
&>& (1-\varepsilon_1)(1-\varepsilon_2)\frac{\delta}{4}(|\lambda_1|+|\lambda_2|).
\end{eqnarray*}
We continue this construction by induction to get a sequence $(P_{N_i}(x_{n_i}))_i$ which is $\frac{4}{\delta(1-\ep)}$-equivalent to the $\ell_1$-basis and verifying $\|P_{N_i}(x_{n_i})-x_{n_i}\| \leq \nolinebreak\varepsilon_i$. By choice we have $C<\varepsilon$. Thus we can apply the principle of small perturbations which gives the following inequalities 
\begin{eqnarray*}
\left \|\sum_{i=1}^m a_i x_{n_i}\right \| &\geq& \frac{1-C}{1+C }\left \|\sum_{i=1}^m a_i P_{N_i}(x_{n_i}) \right \| \\
&\geq& \frac{1-\varepsilon}{1+\varepsilon} (1-\varepsilon)\frac{\delta}{4}  \sum_{i=1}^m |a_i| \\
&\geq& \frac{\delta}{4}\frac{(1-\varepsilon)^2}{1+\varepsilon} \sum_{i=1}^m |a_i|,
\end{eqnarray*}
for every $(a_i)_{i=1}^m \in \R^m$.
This ends the second case and finishes the proof.
\end{proof}

\section{Norm attainment} \label{na}
\bigskip

Given a metric space $M$ and a Banach space $X$, notice that the equality $Lip(M,X)=L(\mathcal F(M),X)$ yields two natural definition of norm attainment for a function $f\in Lip(M,X)$. On the one hand, if we see $f$ as a linear operator, we can consider the classical definition of norm attainment. On the other hand, considering $f\in Lip(M,X)$, we say that $f$ \textit{strongly attains its norm} if there are two different points $x,y\in M$ such that $\Vert f(x)-f(y)\Vert= \Vert f\Vert d(x,y)$. A natural question here is wondering when both concepts of norm-attainment agree and, connected with this, wondering about denseness of the class of Lipschitz functions which strongly attain their norm in $Lip(M,X)$.

We will mean by $Lip_{SNA}(M,X)$ (respectively $NA(\F(M),X)$) to the class of all functions in $Lip(M,X)$ which strongly attain its norm (respectively which attain its norm as a linear and continuous operator from $\mathcal F(M)$ to $X$). Nice results have been recently appeared in this line. On the one hand, negative results can be found in \cite{kadets}, where it is proved that $Lip_{SNA}(X)$ is not dense in $Lip(X)$ whenever $X$ is a Banach space \cite[Theorem 2.3]{kadets}. On the other hand, positive results in this line appear in \cite{godsurvey}, where it is proved that if $M$ is a compact metric space such that $lip(M)$ separates points uniformly and if $E$ is finite dimensional, then $Lip_{SNA}(M,E)$ is norm-dense in $Lip(M,E)$. In the following we will use tensor product theory in order to generalise previous result by considering proper metric spaces as well as by considering more general target spaces.

We shall begin by stating the scalar case of previous result. This can be seen as a generalisation of \cite[Proposition 5.3]{godsurvey}, though we will actually follow the same ideas.

\begin{proposition} \label{reallip}
Let $M$ be a proper metric space such that $S(M)$ separates points uniformly. Then every $f \in Lip(M)$ which attains its norm on $\F(M)$ also strongly attains it. In other words, the following equality holds: 
$$NA(\F(M),\mathbb R)= Lip_{SNA} (M,\mathbb R).$$
Therefore, $\overline{Lip_{SNA}(M,\mathbb R)}^{\|\, \cdot \,\|} = Lip_0(M)$.
\end{proposition}

\begin{proof}
Notice that the inclusion $Lip_{SNA}(M,\mathbb R)\subseteq NA(\mathcal F(M),\mathbb R)$ always holds, so we are just going to prove the reverse one. For this, pick $f\in Lip(M)$ a Lipschitz function which attains its norm as an element of $L(\mathcal F(M),\mathbb R)$. By \cite[Lemma 3.9]{dal2} we get that, for every $\ep >0$, $S(M)$ is $(1+\ep)$-isomorphic to a subspace of $c_0$. Therefore, since the property of being an M-ideal in its bidual is invariant under almost isometric isomorphism and by taking subspaces (see Theorem 3.1.6 and Remark 1.7 in \cite{wernermideal}), the space $S(M)$ is an M-ideal in its bidual. 

Consequently, by \cite[Lemma 5.2]{godsurvey} we get that $f$ attains its norm on some $\gamma \in \F(M) \cap Ext(Lip_0(M)^*)$. But \cite[Corollary 2.5.4]{weaver} implies that $\gamma$ is of the form $\gamma = \lambda \frac{\delta(x)-\delta(y)}{d(x,y)}$ where $x \neq y \in M$ and $\lambda \in \mathbb R$  is such that $|\lambda| = 1$, from where it is clear that $f$ strongly attain its norm. The final assertion is a consequence of Bishop-Phelps theorem.
\end{proof}

We now turn to the study of vector-valued Lipschitz functions. Next result proves that, under assumptions of having a proper metric space $M$, both concepts of norm attainment in $Lip(M,X)$ actually are the same whenever $S(M)$ is a predual of $\mathcal F(M)$.

\begin{proposition} \label{vectorcase}
Let $M$ be a proper metric space such that $S(M)$ separates points uniformly, and let $X$ be a Banach space. Then, for a Lipschitz function $f$: $M \to X$, the following are equivalent:
\begin{enumerate}
\item $f$ strongly attains its norm.
\item $f$: $\F(M) \to X$ attains its operator norm.
\end{enumerate}
Thus the following equality holds: $NA(\F(M),X)= Lip_{SNA} (M,X)$
\end{proposition}

\begin{proof}
We just have to prove $(2) \Rightarrow (1)$ since the other way is always true and trivial. Assume that $\gamma \in \F(M)$ is such that $\|\gamma\| \leq 1$ and $\|f(\gamma) \|= \|f\|_{Lip}$. Then, by Hahn-Banach theorem, there exists $x^* \in S_{X^*} $ verifying $\la x^* , f(\gamma) \ra = \|f(\gamma)\|$. But $x^* \circ f$: $M\to \R$ is a real-valued Lipschitz function which attains its operator norm on $\gamma$. Thus Proposition \ref{reallip} gives the conclusion.
\end{proof}

Since Bishop-Phelps theorem fails in the vector-valued case, we can not deduce directly the same density result as in Proposition \ref{reallip}. However we can state such a density result for a quite big class of Banach spaces using tensor product theory and considering natural assumptions in this frame. First of all we prove the following lemma, which will be useful to get the desired density result for vector-valued Lipschitz functions.

\begin{lemma} \label{lemext}
Let $X$ and $Y$ be Banach spaces such that either $X^*$ or $Y^*$ has (AP) and that $X^*$ and $Y^*$ have both (RNP). If an operator $T$ in \linebreak $L(X^*,Y^{**}) = (X^* \hat\otimes_\pi Y^*)^*$ attains its norm as linear form on $X^* \hat\otimes_\pi Y^*$, then it also attains its norm as an operator on $X^*$. The converse is true when $Y$ is reflexive.
\end{lemma}

\begin{proof}
Assume that $T$ attains its norm as linear form on $X^* \hat\otimes_\pi Y^*$. Since $X^* \hat\otimes_\pi Y^*$ has (RNP) \cite[Theorem VIII.4.7]{vectormeasures}, $T$ attains its norm at some extreme point of the unit ball of this space . Bearing in mind that $Ext(B_{X^* \pten Y^*}) \linebreak = Ext(B_{X^*}) \otimes Ext(B_{Y^*})$ \cite{ruessstegal} we get the existence of $x^* \in Ext(B_{X^*})$ and $y^* \in Ext(B_{Y^*})$ such that 
$$\la y^* , T(x^*) \ra = \la T , x^* \otimes y^* \ra = \|T\|_{(X^* \hat\otimes_\pi Y^*)^*} = \|T\|_{L(X^*,Y^{**})},$$ 
and obviously $\|T(x^*)\|_{Y} = \|T\|_{L(X^*,Y^{**})}$.
 
Conversely we assume that $Y$ is reflexive, so $L(X^*,Y^{**}) = L(X^*,Y)$. Then, if $T$ attains its norm as an operator on $X^*$, there exists $x^* \in S_{X^*}$ such that $\|T(x^*)\|_{Y} = \|T\|_{L(X^*,Y)}$. Now, by Hahn-Banach theorem, there exists $y^* \in B_{Y^{*}}$ such that $\la y^* , T(x^*) \ra = \|T(x^*)\|_{Y}$. Therefore $T$ attains its norm on $x^* \otimes y^* \in X^* \hat\otimes_\pi Y^*$.
\end{proof}

We now state and prove our desired density result for vector-valued Lipschitz functions.

\begin{theorem} \label{vectorna}
Let $(M,d)$ be a proper metric space such that $S(M)$ separates points uniformly. Let $X$ be a Banach space such that $X^*$ has (RNP). Assume that either $\F(M)$ or $X^*$ has (AP). Then 
$$\overline{NA (\F(M),X^{**})}^{\|\, \cdot \, \|} =  L(\F(M),X^{**}).$$ 
Equivalently, $\overline{Lip_{SNA} (M,X^{**})}^{\|\, \cdot \, \|}= Lip(M,X^{**})$.
\end{theorem}

\begin{proof}
Bishop-Phelps theorem applied to $(\F(M) \hat\otimes_\pi X^*)^*= L(\F(M) \hat\otimes_\pi X^* , \mathbb R)$ provides the norm-denseness of the set of those linear forms which attain their norm. But according to Lemma \ref{lemext}, if $f$ attains its norm as linear form on $\F(M) \hat\otimes_\pi X^*$, then $f$ also attains its norm as an operator defined on $\F(M)$. This provides the result. Finally, last assertion follows from Proposition \ref{vectorcase}.
\end{proof}

\section{Some remarks and open questions} \label{rema}
\bigskip

Note that in \cite{dal2} it is proved that, given a proper metric space $M$, then $S(M)$ is $(1+\varepsilon)$-isometric to a subspace of $c_0$ for every $\varepsilon>0$. So, we wonder.

\begin{question}\label{quec0isometric} Let $M$ and $\tau$ be under the hypothesis of Theorem \ref{tauduality}. Is $lip_\tau(M)$ $(1+\varepsilon)$-isometric to a subspace of $c_0$ for every $\varepsilon>0$? 
\end{question}

Note that an affirmative answer to above question would imply that $lip_\tau(M)$ would be an $M$-embedded space. This would have two implications. On the one hand, given a Banach space $X$ such that $\mathcal F(M)$ or $X^*$ has (AP), $lip_\tau(M,X)=lip_\tau(M)\iten X$ would be an unconditionally almost square Banach space \cite{gr} and, consequently, it would not be a dual Banach space, extending Proposition \ref{UASQparanodual}. On the other hand, such a metric space $M$ would satisfy the thesis of Proposition \ref{reallip}.

Another question from Section \ref{schur} is the following.

\begin{question}
Let $M$ be a metric space and let $X$ be a Banach one. If both $\mathcal F(M)$ and $X$ have the Schur property, can we deduce that $\mathcal F(M,X)$ has the Schur property?
\end{question}

Note that an affirmative answer holds for $X=\ell_1(I)$, for any arbitrary set $I$, since 
\[ \mathcal F(M,\ell_1(I))=\mathcal F(M)\widehat{\otimes}_\pi \ell_1(I)=\ell_1(I, \mathcal F(M))\]
has the Schur property if, and only if, $\mathcal F(M)$ has the Schur property \cite[Proposition, Section 2]{tan}. The same conclusion holds whenever $M$ is a proper ultrametric space because in this case $\mathcal F(M)$ is isomorphic to $\ell_1$ \cite{dal2}

\textbf{Acknowledgements}: This work was done when the first and the third authors visited the Laboratoire de Math\'ematiques de Besan\c{c}on, for which the first author was supported by a grant from Programa de Contratos Predoctorales (FPU) de la Universidad de Murcia and the third author was supported by a grant from Vicerrectorado de Internacionalizaci\'on y Vicerrectorado de Investigaci\'on y Transferencia de la Universidad de Granada, Escuela Internacional de Posgrado de la Universidad de Granada y el Campus de Excelencia Internacional (CEI) BioTic. These authors are deeply grateful to the Laboratoire de Math\'ematiques de Besan\c{c}on for their hospitality during the stay.
Moreover, the authors thank to G. Lancien and  A. Proch\'azka for useful conversations which helped to improve the paper.

\end{document}